\documentclass[reqno]{amsart}
\usepackage{amssymb,eucal,latexsym,color,enumerate}
{\end{enumerate}}
{\end{enumerate}}
\newenvironment{enumerater}{\begin{enumerate}[\upshape (1)]}%
{\end{enumerate}}

\newcommand{\pup}[1]{\textup{(}{#1}\textup{)}}
\newcommand{\lgrp}{$\ell$-group}
\newcommand{\lsgrp}{$\ell$-sub\-group}
\newcommand{\lidl}{$\ell$-i\-de\-al}
\newcommand{\lhom}{$\ell$-ho\-mo\-mor\-phism}

\newcommand{\eqdef}{\overset{\mathrm{def}}{=}}
\newcommand{\ci}{\sqsubseteq}
\newcommand{\sci}{\sqsubset}

\DeclareMathOperator{\Max}{Max}
\DeclareMathOperator{\Min}{Min}

\newcommand{\ga}{\alpha}
\newcommand{\gb}{\beta}
\newcommand{\gc}{\gamma}
\newcommand{\gd}{\delta}

\newcommand{\gs}{\sigma}

\newcommand{\go}{\omega}
\newcommand{\eps}{\varepsilon}

\newcommand{\fin}[1]{[{#1}]^{<\omega}}

\newcommand{\gD}{\Delta}
\newcommand{\gL}{\Lambda}
\newcommand{\gO}{\Omega}
\newcommand{\bgO}{\ol{\Omega}}

\newcommand{\sd}{\mathbin{\smallsetminus}}

\newcommand{\dzlat}{distributive $0$-lattice}
\newcommand{\cn}{completely normal}

\newcommand{\ol}[1]{\overline{#1}}

\newcommand{\pI}[1]{\bigl({#1}\bigr)}

\newcommand{\set}[1]{\left\{#1\right\}}
\newcommand{\setm}[2]{\set{{#1}\mid{#2}}}
\newcommand{\vecm}[2]{\left({#1}\mid{#2}\right)}
\newcommand{\seq}[1]{\langle{#1}\rangle}

\newcommand{\op}{\mathrm{op}}
\newcommand{\scp}[2]{({#1}\mid{#2})}

\newcommand{\id}{\mathrm{id}}

\newcommand{\ds}[2]{\operatorname{\downarrow}\nolimits_{#1}{#2}}
\newcommand{\us}[2]{\operatorname{\uparrow}\nolimits_{#1}{#2}}

\newcommand{\es}{\varnothing}
\newcommand{\res}{\mathbin{\restriction}}

\newcommand{\QQ}{\mathbb{Q}}

\newcommand{\conc}{\mathbin{\stackrel{\smallfrown}{}}}

\DeclareMathOperator{\Op}{Op}
\newcommand{\Ops}{\Op^-}

\DeclareMathOperator{\Idc}{Id_c}

\DeclareMathOperator{\Csc}{Cs_c}

\DeclareMathOperator{\Pow}{Pow}

\numberwithin{equation}{section}

\newtheorem*{stat}{\name}
\newcommand{\name}{testing}

\theoremstyle{plain}

\newtheorem{theorem}{Theorem}[section]
\newtheorem{proposition}[theorem]{Proposition}
\newtheorem{corollary}[theorem]{Corollary}
\newtheorem{lemma}[theorem]{Lemma}
\newtheorem{examplepf}[theorem]{Example}
\newtheorem{claim}{Claim}
\newtheorem*{sclaim}{Claim}

\theoremstyle{definition}

\newtheorem{definition}[theorem]{Definition}
\newtheorem{notation}[theorem]{Notation}

\newtheorem{problem}{Problem}

\theoremstyle{remark}
\newtheorem{remark}[theorem]{Remark}
\newtheorem*{note}{Note}

\newcommand{\qedc}{{\qed}~{\rm Claim~{\theclaim}.}}
\newcommand{\qedsc}{{\qed}~{\rm Claim.}}

\newenvironment{cproof}
{\begin{proof}[Proof of Claim.]}
{\qedc\renewcommand{\qed}{}\end{proof}}

\newenvironment{scproof}
{\begin{proof}[Proof of Claim.]}
{\qedsc\renewcommand{\qed}{}\end{proof}}

\numberwithin{figure}{section}
\numberwithin{table}{section}

\newcommand{\bx}{\boldsymbol{x}}

\newcommand{\kk}{\Bbbk}
\newcommand{\kkp}[1]{\kk^{(#1)}}

\newcommand{\fsh}{finitely shadowing}
\newcommand{\Fsh}{Finitely shadowing}

\newcommand{\cA}{\mathcal{A}}
\newcommand{\cB}{\mathcal{B}}

\title{Monotone-Cevian and finitely separable lattices}

\author[M. Plo\v{s}\v{c}ica]{Miroslav Plo\v{s}\v{c}ica}
\address{Faculty of Natural Sciences\\
\v{S}af\'arik's University\\
Jesenn\'a 5\\
04154 Ko\v{s}ice\\
Slovakia}
\email{miroslav.ploscica@upjs.sk}
\urladdr{https://ploscica.science.upjs.sk}

\author[F. Wehrung]{Friedrich Wehrung}
\address{Normandie Universit\'e, UNICAEN\\
CNRS UMR 6139, LMNO\\
14000 Caen\\
France}
\email{friedrich.wehrung01@unicaen.fr}
\urladdr{https://wehrungf.users.lmno.cnrs.fr}

\date{\today}



\begin{document}

\begin{abstract}
A distributive lattice with zero is \emph{\cn} if its prime ideals form a root system under set inclusion.
Every such lattice admits a binary operation $(x,y)\mapsto x\sd y$ satisfying the rules $x\leq y\vee(x\sd y)$ and $(x\sd y)\wedge(y\sd x)=0$ --- in short a \emph{deviation}.
In this paper we study the following additional properties of deviations: \emph{monotone} (i.e., isotone in~$x$ and antitone in~$y$) and \emph{Cevian} (i.e., $x\sd z\leq(x\sd y)\vee(y\sd z)$).
We relate those matters to \emph{finite separability} as defined by Freese and Nation.
We prove that every finitely separable \cn\ lattice has a monotone deviation.

We pay special attention to lattices of principal \lidl{s} of Abelian \lgrp{s} (which are always \cn).
We prove that for free Abelian \lgrp{s} (and also free $\kk$-vector
lattices) those lattices admit monotone Cevian deviations.
On the other hand, we construct an Archimedean \lgrp\ with strong unit whose principal \lidl\ lattice does not have a monotone deviation.
\end{abstract}

\maketitle

\section{Introduction}\label{S:Intro}
This paper is motivated by the investigation of principal \lidl\ lattices of  Abelian \lgrp{s}.
It has been known for a long time that those lattices are distributive with~$0$ and are \cn.
Recall (cf. Wehrung~\cite{MV1}) that a lattice~$D$ is \emph{\cn} if it is distributive, has a least element (usually denoted by~$0$), and for all $a,b\in D$ there are $x,y\in D$ such that $a\vee b=a\vee y=x\vee b$ whereas $x\wedge y=0$.
Equivalently, the prime ideals of~$D$ form a root system under set inclusion (cf. Monteiro~\cite{Mont1954}).
 It is an easy exercise to verify that $D$ is \cn\ if{f} it admits a deviation in the following sense.

\begin{definition}\label{D:Dev}
A binary operation~$\sd$, on a \dzlat~$D$, is a \emph{deviation on~$D$} if the relations $x\leq y\vee(x\sd y)$ and $(x\sd y)\wedge(y\sd x)=0$ both hold whenever $x,y\in D$.
The deviation~$\sd$ is
\begin{itemize}
\item
\emph{left isotone} if $x\leq x'$ implies that $x\sd y\leq x'\sd y$,
\item
\emph{right antitone} if $y\leq y'$ implies that $x\sd y'\leq x\sd y$,
\item\emph{monotone} if it is both left isotone and right antitone;
\item\emph{Cevian} if $x\sd z\leq(x\sd y)\vee(y\sd z)$ whenever $x,y,z\in D$;
\item\emph{monotone-Cevian} if it is both monotone and Cevian.
\end{itemize}
We say that the lattice~$D$ is \emph{Cevian} (resp., \emph{monotone-Cevian}) if it has a Cevian (resp., monotone-Cevian) deviation.
\end{definition}

Any homomorphic image of the principal \lidl\ lattice of an Abelian \lgrp\ is Cevian (cf. Wehrung~\cite{Ceva}).
It is easy to find small \cn\ lattices with a non-monotone or non-Cevian deviation.
However, in many cases the deviation can be``adjusted'' to become monotone and Cevian.
By Plo\v{s}\v{c}ica~\cite{Plo21}, every \cn\ lattice with at most~$\aleph_1$ elements is Cevian
(see also Plo\v{s}\v{c}ica and Wehrung~\cite{MV2}).
This does not extend to the cardinality~$\aleph_2$:
by~\cite{Ceva}, not every \cn\ lattice with~$\aleph_2$ elements is Cevian.

The existence of \emph{monotone} deviations on \cn\ lattices has not been investigated so far.
It is not very difficult to prove that a deviation on a (at most) countable
lattice can be adjusted to become monotone.
In this paper we use the idea of an adjustment to prove the much stronger result that
\emph{every finitely separable \cn\ lattice has a monotone deviation}.
Finite separability is here meant in the sense of Freese and Nation~\cite{FreNat2015}.
As an application, we prove that \emph{principal \lidl\ lattices of free Abelian \lgrp{s} are monotone-Cevian} (cf. Corollary~\ref{C:flg}).

The concept of finite separability, originally invented as one of the conditions that characterize projective lattices, seems to have potential for more applications.
In that direction, we shall establish two equivalent conditions for finite separability in lattices and, more generally, posets.
One of them (viz. Proposition~\ref{P:EquivFSWRC}) states that a poset is finitely separable if{f} it has a \emph{\fsh\ well-ordering}.
The other one (viz. Theorem~\ref{T:CharFS}) states that a poset is finitely separable if{f} it is a ``strong amalgam'' of finite posets over a lower finite poset.

In Section~\ref{S:MonAlt} we prove that any deviation on a distributive $0$-lattice with a finitely shadowing well-ordering can be adjusted to become monotone.
This implies the above-mentioned result (cf. Theorem~\ref{T:DevFS}). 
In Section~\ref{S:DevCX} we turn our attention to (necessarily \cn) lattices that arise as the lattices of all principal \lidl{s} of Abelian \lgrp{s}.
We verify that for free
Abelian \lgrp{s} (more generally, for free $\kk$-vector lattices), those lattices are finitely separable (cf. Proposition~\ref{P:OpEFS}),
which implies, with the help of the Belluce map, that they are monotone-Cevian (cf. Corollaries~\ref{C:fvl} and~\ref{C:flg}).
We also verify, invoking the main result of Plo\v{s}\v{c}ica and Wehrung~\cite{MV2}, that \emph{every finitely separable \cn\ lattice with at most~$\aleph_1$ elements is monotone-Cevian} (cf. Corollary~\ref{C:MonCev2}).

Our final achievement (cf. Section~\ref{S:lgrpCX}) is a construction of \emph{an Archimedean \lgrp\ with strong unit, of cardinality~$\aleph_1$, whose principal \lidl\ lattice does not have any left isotone or right antitone deviation}.
Hence, not only there are
\cn\ lattices without a monotone deviation; they can be constructed as principal \lidl\ lattices of  Abelian \lgrp{s}.
This is rather surprising,
in view of the above-mentioned result for free Abelian \lgrp{s}.

Let us introduce some notation.
Order-preserving maps between posets will be called \emph{isotone}, order-reversing ones are \emph{antitone}.
We will denote by~$\Min{X}$ (resp., $\Max{X}$) the set of all minimal (resp., maximal) elements of a subset~$X$ in a poset~$P$.

\begin{notation}\label{Not:dsus}
For any set~$P$, any $A\subseteq P$, and any binary relation~$\lhd$ on~$P$,
 \begin{align*}
 \ds{\lhd}{A}&\eqdef\setm{x\in P}{(\exists a\in A)(x\lhd a)}\,,\\
 \us{\lhd}{A}&\eqdef\setm{x\in P}{(\exists a\in A)(a\lhd x)}\,.
 \end{align*}
For $a\in P$, we will also write $\ds{\lhd}{a}$ and $\us{\lhd}{a}$ instead of $\ds{\lhd}{\set{a}}$ and $\us{\lhd}{\set{a}}$, respectively.
\end{notation}

A poset~$(P,\leq)$ is \emph{lower finite} if $\ds{\leq}{a}$ is finite whenever $a\in P$.

For any set $X$, $\Pow{X}$ denotes the powerset of~$X$. 
``Countable'' will always mean ``at most countable''.
We will denote by~$\fin{X}$ the set of all finite subsets of a set~$X$.


\section{Finite separability, strong amalgams, and shadows}
\label{S:FSStrAm}

The following definition is stated in Freese and Nation~\cite{FreNat2015}.
However, the condition given in Definition~\ref{D:FinSep} already appeared in Freese and Nation~\cite{FreNat1978}, in connection with projective
lattices.
It was then used in several other works on projectivity.
In Heindorf and Shapiro~\cite{HeiSha1994} the condition was given the name \emph{Freese/Nation property}.
In Fuchino, Koppelberg, and Shelah~\cite{FuKoSh}, it is studied from a set-theoretical point of view.

\begin{definition}\label{D:FinSep}
A poset~$M$ is \emph{finitely separable} if there are functions~$A$ and~$B$ with domain~$M$ such that each~$A(z)$ is a finite set of upper bounds of~$z$, each~$B(z)$ is a finite set of lower bounds of~$z$, and for all $x,y\in M$, $x\leq y$ implies $A(x)\cap B(y)\neq\es$.
Such a pair $(A,B)$ will be called a \emph{separability witness} for~$M$.
\end{definition}

The following deep result is contained in Freese and Nation \cite[Theorem~1]{FreNat2015}.

\begin{theorem}\label{T:FreNat2015}
A \emph{lattice}~$L$ is finitely separable if{f} every lattice homomorphism onto~$L$ has an isotone section.
\end{theorem}

It follows from Freese and Nation \cite[Theorem~6]{FreNat2015} that \emph{every projective member} (thus, in particular, any free member) \emph{of any variety of lattices is finitely separable}.
An example of non-finitely separable lattice is the chain~$\go_1$ of all countable ordinals.
The following result collects a few elementary observations on finite separability.

\begin{proposition}\label{P:BasicFS}\hfill
\begin{enumerater}
\item\label{dualFS}
A poset is finitely separable if{f} its dual poset is finitely separable.

\item\label{ctbleFS}
Every countable poset is finitely separable.

\item\label{RetrFS}
Every order-retract of a finitely separable poset is finitely separable.

\item\label{ConvFS}
Every order-convex subset of a finitely separable poset is finitely separable.

\item\label{bdsFS}
Let $M\sqcup\set{1}$ be the poset obtained by adding a new upper bound~$1$ atop all elements of a poset~$M$.
Then~$M$ is finitely separable if{f} $M\sqcup\set{1}$ is finitely separable.
A similar result holds for the poset $M\sqcup\set{0}$ obtained by adding a new lower bound~$0$.

\item\label{ProdFS}
Any finite product of finitely separable posets is finitely separable.
\end{enumerater}
\end{proposition}

\begin{proof}
\eqref{dualFS} is trivial.
Moreover, as observed on \cite[page~246]{FreNat2015}, \eqref{ctbleFS} is easy.
The argument for~\eqref{RetrFS} is established in the course of the proof of the direction (3)$\Rightarrow$(1) of \cite[Theorem~1]{FreNat2015}.
For any separability witness~$(A,B)$ for a poset~$N$ and any order-convex subset~$M$ of~$N$, $x\mapsto A(x)\cap M$ and $x\mapsto B(x)\cap M$ form a separability witness for~$M$; \eqref{ConvFS} follows.
If $M\sqcup\set{1}$ is finitely separable, then, since~$M$ is an order-convex subset of~$M\sqcup\set{1}$, so is~$M$.
Conversely, for any separability witness $(A,B)$ for~$M$, $x\mapsto A(x)\cup\set{1}$ and $x\mapsto B(x)$ form a separability witness for $M\sqcup\set{1}$; \eqref{bdsFS} follows (\emph{see also \cite[Theorem~10]{FreNat2015} for a more general fact}).
Finally, if~$(A_i,B_i)$ is a separability witness for a poset~$M_i$ whenever $i\in\set{1,2}$, then the maps $(x_1,x_2)\mapsto A_1(x_1)\times A_2(x_2)$ and $(x_1,x_2)\mapsto B_1(x_1)\times B_2(x_2)$ form a separability witness for $M_1\times M_2$; \eqref{ProdFS} follows.
\end{proof}

As we will see in Example~\ref{Ex:SubFS}, the ``order-convex subset'' assumption cannot be replaced by ``subset'' in the statement of Proposition~\ref{P:BasicFS}\eqref{ConvFS}:
that is, a sublattice of a finitely separable lattice need not be finitely separable (cf. Example~\ref{Ex:SubFS}).

\begin{definition}\label{D:WRC}
Let~$M$ be a poset, let $A\subseteq M$, and let $x\in M$.
A subset~$U$ of~$A$ is a \emph{lower shadow} of~$x$ on~$A$ if $A\cap\ds{\leq}{x}=A\cap\ds{\leq}{U}$; \emph{upper shadows} are defined dually.
\end{definition}

\begin{definition}\label{D:Shad}
We say that a subset~$A$ is \emph{\fsh\ in a poset~$M$} if every element of~$M$ has both a finite lower shadow and a finite upper shadow on~$A$.
\end{definition}

Of course, if~$A$ is \fsh\ in~$M$, then for every $x\in M$, the smallest lower shadow (resp., upper shadow) of~$x$, with respect to set inclusion, on~$A$ is $\Max(A\cap\ds{\leq}{x})$ (resp., $\Min(A\cap\us{\leq}{x})$).

\begin{definition}\label{D:StrAmalg}
Let~$P$ be a poset.
We say that a poset~$M$ is the \emph{strong amalgam} of a family $\vecm{M_p}{p\in P}$ of subsets of~$M$ if the following statements hold:
\begin{enumerater}
\item\label{MuuMp}
$M=\bigcup_{p\in P}M_p$;

\item\label{MpIsot}
for all $p\leq q$ in~$P$, $M_p$ is a \fsh\ subset of~$M_q$;

\item\label{IP}(Interpolation Property)
for all $p,q\in P$, all $x\in M_p$, and $y\in M_q$, if $x\leq y$, then there are $r\leq p,q$ in~$P$ and $z\in M_r$ such that $x\leq z\leq y$.

\end{enumerater}
We say that the strong amalgam above is \emph{lower finite} if the poset~$P$ is lower finite.
\end{definition}

Items~\eqref{MpIsot} and~\eqref{IP} together obviously entail the following:
 \begin{equation*}
 M_p\cap M_q=\bigcup\setm{M_r}{r\leq p,q}\,,\quad
 \text{whenever }p,q\in P\,.
 \end{equation*}

\begin{proposition}\label{P:StrAmFS}
The following statements hold, for any lower finite strong amalgam~$M$ of a family $\vecm{M_p}{p\in P}$ of subsets:
\begin{enumerater}
\item\label{Shad}
Each~$M_p$ is \fsh\ in~$M$.

\item\label{LFFS}
If each~$M_p$ is finitely separable, then so is~$M$.

\end{enumerater}
\end{proposition}

\begin{proof}
\emph{Ad}~\eqref{Shad}
We must prove that every $x\in M$ has (say) an upper shadow on each~$M_q$.
Pick $p\in P$ such that $x\in M_p$.
Since~$P$ is lower finite, $R\eqdef\ds{\leq}{p}\cap\ds{\leq}q$ is finite.
For each $r\in R$, it follows from Definition~\ref{D:StrAmalg}\eqref{MpIsot} that~$M_r$ is \fsh\ in~$M_p$; thus~$x$ has a finite upper shadow~$U_r$ on~$M_r$.
A direct application of the Interpolation Property (cf. Definition~\ref{D:StrAmalg}\eqref{IP}) then shows that $\bigcup_{r\in R}U_r$ is a finite upper shadow of~$x$ on~$M_q$.

\emph{Ad}~\eqref{LFFS}.
For each $x\in M$, pick $\nu(x)\in P$ such that $x\in M_{\nu(x)}$.
For every $p\in P$, pick a separability witness $(A_p,B_p)$ for~$M_p$.
For all $x\in M$ and all $p\leq\nu(x)$, it follows from Definition~\ref{D:StrAmalg}\eqref{MpIsot} that~$x$ has a finite upper shadow~$U_{x,p}$ and a finite lower shadow~$V_{x,p}$ on~$M_p$.
The sets
 \begin{align*}
 A(x)&\eqdef
 \bigcup\setm{A_p(u)}{p\leq\nu(x)\,,\ u\in U_{x,p}}\,,\\
 B(x)&\eqdef
 \bigcup\setm{B_p(v)}{p\leq\nu(x)\,,\ v\in V_{x,p}}
 \end{align*}
are, respectively, a finite set of upper bounds of~$x$ and a finite set of lower bounds of~$x$ in~$M$.
Let $x\leq y$ in~$M$; set $p\eqdef\nu(x)$ and $q\eqdef\nu(y)$.
By the Interpolation Property (cf. Definition~\ref{D:StrAmalg}\eqref{IP}), there are $r\in\ds{\leq}{p}\cap\ds{\leq}{q}$ and $z\in M_r$ such that $x\leq z\leq y$.
By definition, there are $u\in U_{x,r}$ and $v\in V_{y,r}$ such that $x\leq u\leq z\leq v\leq y$.
Since $u\leq v$ within~$M_r$, there exists $w\in A_r(u)\cap B_r(v)$; so $x\leq w\leq y$ whereas $w\in A(x)\cap B(y)$.
Therefore, $(A,B)$ is a separability witness for~$M$.
\end{proof}

\begin{examplepf}\label{Ex:SubFS}
A finitely separable lattice with a non-finitely separable sublattice.
\end{examplepf}

\begin{proof}
The lattice $P\eqdef\fin{\go_1}$ is the strong amalgam of its finite sublattices $\fin{X}$, for finite $X\subset\go_1$; thus, by Proposition~\ref{P:StrAmFS}, it is finitely separable, and thus so is its opposite lattice~$P^{\op}$.
It follows from Freese and Nation \cite[Lemma~9]{FreNat2015} that the ordinal sum $P\dotplus P^{\op}$ is not finitely separable.

Denote by~$u$ a new top element for~$P$ (thus also a new bottom element for~$P^{\op}$).
It follows from Proposition~\ref{P:BasicFS}\eqref{bdsFS} that~$P\cup\set{u}$ and~$P^{\op}\cup\set{u}$ are both finitely separable, thus so is their product $(P\cup\set{u})\times(P^{\op}\cup\set{u})$.
Moreover, $P\dotplus P^{\op}$ embeds into $(P\cup\set{u})\times(P^{\op}\cup\set{u})$, by mapping each $x\in P$ to $(x,u)$ and each $y\in P^{\op}$
 to $(u,y)$.
\end{proof}

\section{\Fsh\ well-orderings}
\label{S:WRCWO}

A typical situation that will arise in this section will involve two distinct orderings~$\leq$ and~$\ci$ on the same universe~$M$, occasionally prompting the need to spell out which one is in question.
For example, Definition~\ref{D:WRCWO} will begin with ``Let $(M,\leq)$ be a poset'' instead of ``Let~$M$ be a poset''.

\begin{definition}\label{D:WRCWO}
Let $(M,\leq)$ be a poset.
A binary relation~$\ci$ of~$M$ is \emph{\fsh} on~$(M,\leq)$ if~$\ds{\ci}{x}$ is \fsh\ in~$(M,\leq)$ whenever $x\in M$.
\end{definition}

If~$\ci$ is a partial ordering, with associated strict ordering~$\sci$, then~$\ds{\ci}{x}$ is \fsh\ if{f}~$\ds{\sci}{x}$ is \fsh\ (for these two sets differ by the singleton~$\set{x}$).
In particular, $\ci$ is \fsh\ if{f}~$\sci$ is \fsh.

\begin{proposition}\label{P:WRCWOchar}
Let~$(M,\leq)$ be a poset.
Then a well-ordering~$\ci$ on~$M$ is \fsh\ in $(M,\leq)$ if{f} every $a\in M$ has both a finite upper shadow and a finite lower shadow on~$\ds{\sci}{a}$.
\end{proposition}

\begin{proof}
We verify the nontrivial direction.
Suppose that the given condition holds and let $a,b\in M$; we must verify that~$b$ has (say) a finite upper shadow on~$\ds{\sci}{a}$.
We argue by $\ci$-induction on~$b$.
The result is trivial if $b\sci a$ (for then $b\in\ds{\sci}{a}$); we may thus suppose that $a\ci b$.
By assumption, $b$ has a finite upper shadow~$A$ on~$\ds{\sci}{b}$.
By induction hypothesis, every $x\in\ds{\sci}{b}$ has a finite upper shadow~$U_x$ on~$\ds{\sci}{a}$.
Then $\bigcup_{x\in A}U_x$ is a finite upper shadow of~$b$ on~$\ds{\sci}{a}$.
\end{proof}

\begin{proposition}\label{P:EquivFSWRC}
A poset~$(M,\leq)$ is finitely separable if{f} it has a \fsh\ well-ordering.
Furthermore, for every \fsh\ well-ordering~$\ci$ on~$M$, there exists a separability witness $(A,B)$ of~$M$ such that
 \begin{equation}\label{Eq:ABthrowlow}
 x\in A(y)\cup B(y)\text{ implies that }x\ci y\,,\quad
 \text{for all }x,y\in M\,.
 \end{equation}
\end{proposition}

\begin{proof}
The argument of the proof that every finitely separable poset has a \fsh\ well-ordering is mostly contained in the proof of Freese and Nation \cite[Theorem~1]{FreNat2015}.
For convenience, we provide a description of the well-ordering.
Let $(A,B)$ be a separability witness for~$M$.
We define inductively an ordinal~$\gd$ and a partition $\vecm{M_{\xi}}{\xi<\gd}$ of~$M$ into countable blocks, as follows.
Suppose $\vecm{M_{\xi}}{\xi<\ga}$ already defined and set $M_{<\ga}\eqdef\bigcup_{\xi<\ga}M_{\xi}$.
If $M_{<\ga}=M$ then set $\gd\eqdef\ga$ and stop.
Suppose that $M_{<\ga}\neq M$ and pick $c\in M\setminus M_{<\ga}$.
The smallest subset~$M_{\ga}$ of~$M$ such that $c\in M_{\ga}$ and $(A(x)\cup B(x))\setminus M_{<\ga}\subseteq M_{\ga}$ whenever $x\in M_{\ga}$ is countable, and disjoint from~$M_{<\ga}$.
This completes the induction step.

For any $x\in M$, denote by~$\nu(x)$ the unique~$\xi<\gd$ such that $x\in M_{\xi}$.
Pick a well-ordering~$\ci_{\xi}$ of~$M_{\xi}$ of type at most~$\go$, for each $\xi<\ga$, and let
 \[
 x\ci y\quad\text{if}\quad\pI{\nu(x)<\nu(y)\text{ or }
 (\nu(x)=\nu(y)\text{ and }x\ci_{\nu(x)}y)}\,,\qquad
 \text{for all }x,y\in M\,.
 \]
Then~$\ci$ is a well-ordering of~$M$ and for any $c\in M$, with $\gc\eqdef\nu(c)$, $A(c)\cap M_{<\gc}$ is a finite upper shadow and $B(c)\cap M_{<\gc}$ is a finite lower shadow of~$c$ on~$M_{<\gc}$.
Since $\ds{\sci}{c}=M_{<\gc}\cup F$ for the finite set $F\eqdef\setm{x\in M_{\gc}}{x\sci_{\gc}c}$, $c$ also has a finite upper shadow and a finite lower shadow on~$\ds{\sci}{c}$.
Hence, the well-ordering~$\ci$ is \fsh\ in~$(M,\leq)$.

Conversely, let~$\ci$ be a \fsh\ well-ordering on~$M$.
For $c\in M$, we shall define~$A(c)$ and~$B(c)$ by $\ci$-induction, in such a way that $A(c)\cup B(c)\subseteq\ds{\ci}{c}$; this will ensure~\eqref{Eq:ABthrowlow}.
By induction hypothesis, $c$ has a finite upper shadow~$U_c$ and a finite lower shadow~$V_c$ on~$\ds{\sci}{c}$.
Then $A(c)\eqdef\set{c}\cup\bigcup_{u\in U_c}A(u)$ is a finite set of upper bounds of~$c$, and for each $u\in U_c$, $A(u)\subseteq\ds{\ci}{u}$ by induction hypothesis, so $A(u)\subseteq\ds{\sci}{c}$; whence $A(c)\subseteq\ds{\ci}{c}$.
Symmetrically, $B(c)\eqdef\set{c}\cup\bigcup_{v\in V_c}B(v)$ is a finite set of lower bounds of~$c$ contained in~$\ds{\ci}{c}$.

We claim that $(A,B)$ is a separability witness for~$(M,\leq)$.
Let $a,b\in M$ such that $a\leq b$, we verify that $A(a)\cap B(b)\neq\es$.
We argue by $\ci$-induction with respect to $\max_{\ci}\set{a,b}$ (i.e., the maximum of~$\set{a,b}$ with respect to~$\ci$).
Since each $c\in A(c)\cap B(c)$, we may assume that $a<b$.
If $a\sci b$, then $a\leq v$ for some $v\in V_b$, so, since $\set{a,v}\subseteq\ds{\sci}{b}$ and by our induction hypothesis, $A(a)\cap B(v)\neq\es$, and so, since $B(v)\subseteq B(b)$, we get $A(a)\cap B(b)\neq\es$.
The argument for $b\sci a$ is symmetric.
This completes the proof of our claim.
\end{proof}

\begin{definition}\label{D:LocFinSetMap}
For a  map $C\colon M\to\Pow{M}$, we set $C^0(x)=\set{x}$ and $C^{n+1}(x)\eqdef\bigcup_{y\in C(x)}C^n(y)$ whenever $n<\go$.
We say that the map~$C$ is \emph{locally finite} if $C^{\go}(x)\eqdef\bigcup_{n<\go}C^n(x)$ is finite whenever $x\in M$.
\end{definition}

\begin{proposition}\label{P:FS2LocFin}
Every finitely separable poset~$M$ has a separability witness $(A,B)$ such that the set map $(A\cup B\colon x\mapsto A(x)\cup B(x))$ is locally finite.
\end{proposition}

\begin{proof}
By Proposition~\ref{P:EquivFSWRC}, $M$ has a \fsh\ well-ordering~$\ci$ and a separability witness $(A,B)$ satisfying~\eqref{Eq:ABthrowlow}.
If is straightforward to verify, by $\ci$-induction on~$x$, that $(A\cup B)^{\go}(x)$ is finite whenever $x\in M$.
\end{proof}

\begin{theorem}\label{T:CharFS}
A poset~$M$ is finitely separable if{f} it is the strong amalgam, over a lower finite poset \pup{resp., a sublattice of~$\fin{M}$}, of a family of nonempty finite subsets.
\end{theorem}

\begin{proof}
One direction, that any lower finite strong amalgam of finite posets is finitely separable, is provided by Proposition~\ref{P:StrAmFS}.

Let, conversely, $M$ be a finitely separable poset.
By Proposition~\ref{P:FS2LocFin}, $M$ has a separability witness $(A,B)$ such that the set map $A\cup B$ is locally finite.
We say that a subset~$X$ of~$M$ is \emph{closed} if $A(x)\cup B(x)\subseteq X$ whenever $x\in X$.
Fix $o\in M$.
Since the map $A\cup B$ is locally finite, every finite subset of~$M$ is contained in some finite closed subset of~$M$, thus the collection~$\gL$ of all finite closed subsets of~$M$ containing~$\set{o}$ is a sublattice of $(\fin{M},\cup,\cap)$ with $M=\bigcup\gL$.
Now let $P,Q\in\gL$ and let $x\in P$, $y\in Q$ such that $x\leq y$.
Since $(A,B)$ is a separability witness for~$M$, there exists $z\in A(x)\cap B(y)$.
Since~$P$ and~$Q$ are both closed, $z$ belongs to~$P\cap Q$.
Since $x\leq z\leq y$, this yields the Interpolation Property (cf. Definition~\ref{D:StrAmalg}\eqref{IP}) for~$\gL$.
The condition \ref{D:StrAmalg}\eqref{MpIsot} follows from the finiteness of all members of~$\gL$.
\end{proof}

\section{The monotone adjustment of a map}
\label{S:MonAlt}

Now we turn our attention to monotone deviations, the original motivation.

\begin{definition}\label{D:Monot}
Let~$M$ and~$L$ be posets.
A map $d\colon M\times M\to L$ is \emph{monotone} on a subset~$Z$ of~$M\times M$ if $x\leq x'$ and $y'\leq y$ implies that $d(x,y)\leq d(x',y')$ whenever $(x,y),(x',y')\in Z$.
\end{definition}

The proof of the following lemma is routine and we omit it.

\begin{lemma}\label{L:unlhd}
For any elements~$x$, $y$, $x'$, $y'$ in a chain~$(M,\ci)$,  let $\set{x,y}\unlhd\set{x',y'}$ hold if either
$\max\nolimits_{\ci}\set{x,y}\sci\max\nolimits_{\ci}\set{x',y'}$ or \pup{$\max\nolimits_{\ci}\set{x,y}=\max\nolimits_{\ci}\set{x',y'}$ and $\min\nolimits_{\ci}\set{x,y}\ci\min\nolimits_{\ci}\set{x',y'}$}.
Then $\unlhd$ is a  total ordering on $M_{1,2}=\{N\subseteq M\mid |N|\in\{1,2\}\}$ and the following statements hold whenever $x,y,x',y',z\in M$:
\begin{enumerater}
\item
$x\ci y$ if{f} $\{x,z\}\unlhd\{y,z\}$ and $x\sci y$ if{f} $\{x,z\}\lhd\{y,z\}$.

\item
\pup{$x\ci x'$ and $y\ci y'$} implies that $\set{x,y}\unlhd\set{x',y'}$.

\item
$\set{x,y}\lhd\set{x',y'}$ implies that  $x\sci x'$ or $y\sci y'$.

\item 
If $(M,\ci)$ is well-ordered then so is $(M_{1,2},\unlhd)$. 

\end{enumerater}

\end{lemma}

The following technical lemma is the key point to our forthcoming definition of the monotone adjustment of a map.

\begin{lemma}\label{L:DefMonAlt}
Let~$M$ and~$L$ be posets, let $d\colon M\times M\to L$, and let~$\ci$ be a \fsh\ total ordering on~$M$.
Denote by~$\unlhd$ the total ordering on~$M_{1,2}$ introduced in Lemma~\textup{\ref{L:unlhd}}.
Let $a,b\in M$ and suppose that~$d$ is monotone on $D_{a,b}\eqdef\setm{(x,y)\in M\times M}{\set{x,y}\lhd\set{a,b}}$.
For each $x\in\set{a,b}$, let~$U_x$ \pup{resp., $V_x$} be an upper shadow \pup{resp., lower shadow} of~$x$ on~$\ds{\sci}{x}$.
We set
 \begin{align*}
 \cA&\eqdef
 \setm{d(x,y)}{\set{x,y}\lhd\set{a,b}\,,\ a\leq x\,,\ y\leq b}\,,\\
 \cA'&\eqdef
 \setm{d(x,b)}{x\in U_a}\cup\setm{d(a,y)}{y\in V_b}\,,\\
 \cB&\eqdef
 \setm{d(x,y)}{\set{x,y}\lhd\set{a,b}\,,\ x\leq a\,,\ b\leq y}\,,\\
 \cB'&\eqdef
 \setm{d(x,b)}{x\in V_a}\cup\setm{d(a,y)}{y\in U_b}\,.
 \end{align*}
Then~$\cA'$ is a coinitial subset of~$\cA$ and~$\cB'$ is a cofinal subset of~$\cB$, both finite if~$U_a$, $V_a$, $U_b$, $V_b$ are finite.
\end{lemma}

\begin{proof}
We establish for example the part about~$\cA$ and~$\cA'$; the proof for~$\cB$ and~$\cB'$ is similar.
The containment $\cA'\subseteq\cA$ follows immediately from Lemma~\ref{L:unlhd}.
Now let $\set{x,y}\lhd\set{a,b}$, $a\leq x$, $y\leq b$.
By Lemma~\ref{L:unlhd}, either $x\sci a$ or $y\sci b$.
In the first case, there exists $u\in U_a$ such that $u\leq x$.
Since $\{u,b\}\lhd\set{a,b}$ (cf. Lemma~\ref{L:unlhd}), $\set{x,y}\lhd\set{a,b}$, $u\leq x$, and $y\leq b$, it follows that $d(x,y)\geq d(u,b)$, with $d(u,b)\in\cA'$.
In the second case, there exists $v\in V_b$ such that $y\leq v$.
By a similar argument to the above, $d(x,y)\geq d(a,v)$ with $d(a,v)\in\cA'$.
\end{proof}

\begin{definition}
For any poset~$L$ and any $X\subseteq L$, say that an element~$a$ is the \emph{finitary join} of~$X$ if there exists a finite cofinal subset~$X'$ of~$X$ such that $a=\bigvee X'$ \pup{of course in such a case $a=\bigvee X$ as well}.
Finitary meets are defined dually.
\end{definition}

\begin{proposition}\label{P:ConstrMonAlt}
Let~$M$ be a poset, let~$D$ be a lattice, let $d\colon M\times M\to D$, and let~$\ci$ be a \fsh\ well-ordering on~$M$.
Denote by~$\unlhd$ the well-ordering on~$M_{1,2}$ introduced in Lemma~\textup{\ref{L:unlhd}}.
Then there exists a \pup{necessarily unique} monotone map $d'\colon M\times M\to D$ such that for all $a,b\in M$, $d'(a,b)=d'_{\land}(a,b)\vee d'_{\lor}(a,b)$ where
 \begin{align}
 d'_{\land}(a,b)&\eqdef d(a,b)\wedge\bigwedge
 \setm{d'(x,y)}{\set{x,y}\lhd\set{a,b}\,,\ a\leq x\,,\ y\leq b}\,,
 \label{Eq:d'land}\\
 d'_{\lor}(a,b)&\eqdef\bigvee
 \setm{d'(x,y)}{\set{x,y}\lhd\set{a,b}\,,\ x\leq a\,,\ b\leq y}
 \label{Eq:d'lor}
 \end{align}
\pup{where as usual, the empty meet is a new top element and the empty join is a new bottom element}
are a finitary meet and a finitary join, respectively.
\end{proposition}

\begin{proof}
We argue by $\unlhd$-induction. We prove that, for every $a,b\in M$,  $d'$  is correctly defined and monotone on 
$D_{a,b}\cup\set{(a,b),(b,a)}$. Our induction hypothesis implies that~$d'$ is correctly defined and monotone on~$D_{a,b}$.
By Lemma~\ref{L:DefMonAlt} applied to the restriction of~$d'$ to~$D_{a,b}$, the meet $\bigwedge
 \setm{d'(x,y)}{\set{x,y}\lhd\set{a,b}\,,\ a\leq x\,,\ y\leq b}$ (i.e., the second meetand of~\eqref{Eq:d'land}) and the join~\eqref{Eq:d'lor} are both finitary, so $d'_{\land}(a,b)$, $d'_{\lor}(a,b)$, $d'(a,b)$, as well as
 $d'_{\land}(b,a)$, $d'_{\lor}(b,a)$, $d'(b,a)$ are all well defined.

It remains to verify that~$d'$ is monotone on~$D_{a,b}\cup\set{(a,b),(b,a)}$.
This verification splits into several cases.

First, we need to argue that for all $(x,y)\in D_{a,b}$, $x\leq a$ and $b\leq y$ implies $d'(x,y)\leq d'(a,b)$.
This is obvious, because $d'(x,y)$ is then a joinand of $d'_{\lor}(a,b)$. The same argument applies when $a$ and $b$ 
are interchanged.

The second case consists of verifying that for all $(x,y)\in D_{a,b}$, $a\leq x$ and $y\leq b$ implies that $d'(a,b)\leq d'(x,y)$; that is, $d'_{\land}(a,b)\leq d'(x,y)$ and $d'_{\lor}(a,b)\leq d'(x,y)$.
The first inequality is obvious, because $d'(x,y)$ is then a meetand of $d'_{\land}(a,b)$.
In order to prove the second inequality, we must verify that for any $(u,v)\in D_{a,b}$, $u\leq a$ and $b\leq v$ implies that $d'(u,v)\leq d'(x,y)$. This follows from our induction hypothesis, because $u\leq a\leq x$ and $y\leq b\leq v$.
The same argument applies when $a$ and $b$ 
are interchanged.

Finally, we need to prove that $d'(a,b)\leq d'(b,a)$ when $a<b$. If $a\ci b$ then $(a,a)\in D_{a,b}$ and we have already proved that 
$d'(a,b)\leq d'(a,a)\leq d'(b,a)$. If $b\ci a$ then $(b,b)\in D_{a,b}$ and we have already proved that 
$d'(a,b)\leq d'(b,b)\leq d'(b,a)$.
\end{proof}

Owing to Proposition~\ref{P:ConstrMonAlt}, the map~$d'$ will be called the \emph{monotone adjustment} of~$d$.
It depends not only of~$d$, but also of the chosen \fsh\ well-ordering~$\ci$ of~$M$.
Note that~$d$ is monotone if{f} $d=d'$.

Our next two propositions will entail that if~$d$ is a deviation, then so is~$d'$.
Because of possible future applications, we prove them under slightly more general assumptions.

\begin{proposition}\label{P:MonAlt2Iso}
Let~$M$ be a poset, let~$D$ be a distributive lattice, let\newline $d\colon M\times M\to D$, and let $f\colon M\to D$ be an isotone map.
We denote by~$d'$ the monotone adjustment of~$d$ with respect to a \fsh\ well-ordering~$\ci$ of~$M$.
If $f(x)\leq f(y)\vee d(x,y)$ whenever $x,y\in M$, then $f(a)\leq f(b)\vee d'(a,b)$ whenever $a,b\in M$.
\end{proposition}

\begin{proof}
We argue by $\unlhd$-induction. Let $a,b\in M$ and suppose that the inequality $f(x)\leq f(y)\vee d'(x,y)$ holds whenever $\set{x,y}\lhd\set{a,b}$.
In order to prove that\linebreak $f(a)\leq f(b)\vee d'(a,b)$
it suffices to verify that $f(a)\leq f(b)\vee d'_{\land}(a,b)$.
Since the meet defining~$d'_{\land}(a,b)$ is finitary, it follows from the distributivity of~$D$ that it suffices to verify that for all 
$(x,y)\in D_{a,b}$, $a\leq x$ and $y\leq b$ implies that $f(a)\leq f(b)\vee d'(x,y)$.
Since~$f$ is isotone and by our induction hypothesis, we get the inequalities
 \begin{equation*}
 f(a)\leq f(x)\leq f(y)\vee d'(x,y)\leq f(b)\vee d'(x,y)\,.
 \tag*{\qed}
 \end{equation*}
\renewcommand{\qed}{}
\end{proof}

Say that a $0$-lattice~$D$ is \emph{$0$-distributive} if $x\wedge z=y\wedge z=0$ implies $(x\vee y)\wedge z=0$ whenever $x,y,z\in D$.

\begin{proposition}\label{P:MonAlt2Norm}
Let~$M$ be a poset, let~$D$ be a $0$-distributive $0$-lattice, and let \newline $d\colon M\times M\to D$.
We denote by~$d'$ the monotone adjustment of~$d$ with respect to a \fsh\ well-ordering~$\ci$ of~$M$.
If $d(x,y)\wedge d(y,x)=0$ whenever $x,y\in M$, then $d'(a,b)\wedge d'(b,a)=0$ whenever $a,b\in M$.
\end{proposition}

\begin{proof}
We argue by $\unlhd$-induction.
Let $a,b\in M$ and suppose that  $d'(x,y)\wedge d'(y,x)=\nobreak0$ whenever $\set{x,y}\lhd\set{a,b}$.
Since~$D$ is $0$-distributive, the proof breaks up into four statements.
\begin{itemize}
\item
$d'_{\land}(a,b)\wedge d'_{\land}(b,a)=0$.
Since $d'_{\land}(a,b)\leq d(a,b)$, $d'_{\land}(b,a)\leq d(b,a)$, and $d(a,b)\wedge d(b,a)=0$, this is obvious.

\item
$d'_{\land}(a,b)\wedge d'_{\lor}(b,a)=0$.
Since~$D$ is $0$-distributive, it suffices to prove that for any $(x,y)\in D_{a,b}$, $a\leq x$ and $y\leq b$ implies that $d'_{\land}(a,b)\wedge d'_{\lor}(y,x)=0$.
Since $d'(x,y)$ is a meetand of~$d'_{\land}(a,b)$, we get $d'(a,b)\leq d'(x,y)$.
By our induction hypothesis, $d'(x,y)\wedge d'(y,x)=0$, so we are done.

\item
$d'_{\lor}(a,b)\wedge d'_{\land}(b,a)=0$.
This case is symmetric to the case above.

\item
$d'_{\lor}(a,b)\wedge d'_{\lor}(b,a)=0$.
Since~$D$ is $0$-distributive, it suffices to prove that for any $(x,y),(u,v)\in D_{a,b}$, $x\leq a\leq u$ and $v\leq b\leq y$ implies that $d'(x,y)\wedge d'(v,u)=0$.
Since~$\ci$ is a total ordering, either $v\ci y$ or $y\ci v$.
In the first case, then (cf. Lemma~\ref{L:unlhd}) $\set{x,v}\unlhd\set{x,y}\lhd\set{a,b}$, thus, by our induction hypothesis, $d'(x,v)\wedge d'(v,x)=0$.
Since $d'(x,y)\leq d'(x,v)$ and $d'(v,u)\leq d'(v,x)$, the desired conclusion follows.
In the second case, then (cf. Lemma~\ref{L:unlhd}) $\set{u,y}\unlhd\set{u,v}\lhd\set{a,b}$, thus, by our induction hypothesis, $d'(u,y)\wedge d'(y,u)=0$.
Since $d'(x,y)\leq d'(u,y)$ and $d'(v,u)\leq d'(y,u)$, the desired conclusion follows.\qed
\end{itemize}
\renewcommand{\qed}{}
\end{proof}

By taking~$M=D$, $f=\id_D$, and~$d$ any deviation on~$D$ in Propositions~\ref{P:ConstrMonAlt}, \ref{P:MonAlt2Iso}, and~\ref{P:MonAlt2Norm}, we obtain immediately the following.

\begin{theorem}\label{T:DevFS}
Every finitely separable \cn\ lattice has a monotone deviation.
\end{theorem}

\section{Lattices of principal \lidl{s} in  Abelian \lgrp{s}}
\label{S:DevCX}

We will always denote \lgrp{s} additively and set $|a|\eqdef a\vee(-a)$ whenever $a\in G$.
The convex \lsgrp\ generated by~$a$ is
 \[
 \seq{a}\eqdef\{x\in G\mid -n|a|\leq x\leq n|a|\
 \text{for some}\ n\in\omega\}\,.
 \]
Clearly, for any $a,b\in G^+$, $\seq{a}\subseteq\seq{b}$ if{f} $a\leq mb$ for some $m\in\go$.
All these convex \lsgrp{s} form a \cn\ lattice~$\Csc{G}$, often denoted~$\Idc{G}$ if~$G$ is Abelian (for in that case, \lidl{s} and convex \lsgrp{s} are the same).
A deviation on~$\Csc{G}$ can be defined by choosing a positive generator for each convex \lsgrp\ and setting
 \[
 \seq{a}\sd\seq{b}\eqdef\seq{(a-b)^+}\,,
 \]
where  $x^+\eqdef x\vee 0$.
Of course, this operation depends on the choice of the generators~$a$, $b$.
It is always Cevian (cf. \cite[Proposition~5.5]{Ceva}).
It will be monotone if we could choose the generators in an isotone way, that is  $\seq{a}\leq\seq{b}$ implies $a\leq b$.
This argument yields the following observation.

\begin{theorem}\label{T:MonCev}
Let~$D$ be a \dzlat\ and suppose that there exists a surjective lattice homomorphism $f\colon\Csc{G}\twoheadrightarrow D$ for some \lgrp~$G$.
If~$D$ is finitely separable, then it is monotone-Cevian.
\end{theorem}

\begin{note}
Since every homomorphic image of a \cn\ lattice is \cn, our assumptions imply that~$D$ is \cn.
\end{note}

\begin{proof}
The map $\gb\colon G^+\twoheadrightarrow\Csc{G}$, $x\mapsto\seq{x}$ ($\gb$ is sometimes called the \emph{Belluce map}), is a surjective lattice homomorphism (cf. Bigard \emph{et al.} \cite[Proposition~2.2.11]{BKW}).
Since~$D$ is finitely separable, by Theorem~\ref{T:FreNat2015} the composite map $f\circ\gb$ has an isotone section $\gs\colon D\hookrightarrow\nobreak G^+$.
The assignment $(x,y)\mapsto f\seq{(\gs(x)-\gs(y))^+}$ defines a monotone-Cevian operation on~$D$.
\end{proof}

\begin{remark}\label{Rk:MonCev}
Theorem~\ref{T:MonCev} trivially extends to finitely separable homomorphic images of \lidl\ lattices of vector lattices.
Indeed, for any vector lattice~$E$ over a totally ordered division ring~$\kk$, with underlying  Abelian \lgrp~$G$, the assignment $\seq{x}\mapsto\kk\cdot\seq{x}$ defines a surjective lattice homomorphism $\Idc{G}\twoheadrightarrow\Idc{E}$ (which is an isomorphism if~$\kk$ is Archimedean).
\end{remark}

The main result of Plo\v{s}\v{c}ica and Wehrung~\cite{MV2} states that \emph{every \cn\ lattice with at most~$\aleph_1$ elements is a homomorphic image of~$\Idc{G}$ for some  Abelian \lgrp~$G$}.
Hence, we get the following variant of Plo\v{s}\v{c}ica \cite[Theorem~3.2]{Plo21}, obtained by strengthening there both the assumption (the lattice is now assumed to be finitely separable) and the conclusion (``Cevian'' becomes ``monotone-Cevian'').

\begin{corollary}\label{C:MonCev2}
Every finitely separable \cn\ lattice with at most~$\aleph_1$ elements is monotone-Cevian.
\end{corollary}

By Theorem~\ref{T:MonCev}, if $\Idc G$ is finitely separable, then it is monotone-Cevian.
In the sequel we shall verify that this is the case for free Abelian \lgrp{s}.
We find it convenient to work in the slightly more general setting of vector lattices.
Let~$\kk$ be a countable totally ordered division ring and let  $E=\kkp{I}$ be the left vector space over $\kk$ with basis $I$.
For $J\subseteq I$ we identify~$\kkp{J}$ with its canonical copy in $E$.
For every  $a\in E$ we set
 \begin{equation*}
 [a]\eqdef\setm{x\in E}{\scp{a}{x}>0}\,, \end{equation*}
(where $(a\mid x)\eqdef\sum_{i\in I}a_ix_i$).
Following Wehrung \cite{MV1,RAlg}, we  denote by~$\Ops\kkp{J}$ the $0$-sublattice of the powerset of~$E$ generated by $\setm{[x]}{x\in\kkp{J}}$.
In other words, $\Ops\kkp{J}$ is the lattice of
all subsets of $E$ defined as disjunctions of conjunctions of strict  linear inequalities with variables in $J$.
Further, let $\Op\kkp{J}=\Ops\kkp{J}\cup\set{E}$.
(Notice that $E\notin\Ops\kkp{J}$, so we are adding a new top element.)
With the above conventions, $\Op\kkp{J}$ is thus identified with a $0$-sublattice of $\Op\kkp{I}$.

\begin{lemma}\label{shad}
For every $X\in\fin{I}$, the lattice $\Op\kkp{X}$ is a \fsh\ subset of  $\Op E$.
\end{lemma}

 \begin{proof}
(See also the proof of Plo\v{s}\v{c}ica and Wehrung~\cite[Lemma~12]{MV2}.)
We claim that the upper and lower shadow of any $U\in\Op E$ are, respectively, $\set{U^*}$ and~$\set{U_*}$, where
\begin{align*}
U^*&\eqdef\setm{v\in E}{v\res_X=u\res_X\ \text{for some}\ u\in U}\,,\\
U_*&\eqdef\setm{v\in E}{u\in U\ \text{whenever}\  v\res_X=u\res_X}\,.
\end{align*}
First, the above defined sets are both elements of $\Op\kkp{X}$:
this follows from quantifier elimination for the theory of all nontrivial totally ordered $\kk$-vector spaces (cf. van den Dries \cite[Corollary~I.7.8]{vdD1998}).

Clearly, $U_*\subseteq U\subseteq U^*$. Now let $Z\in\Op\kkp{X}$, $U\subseteq Z$.
We verify that $U^*\subseteq Z$.
It suffices to consider the case when
$Z=\bigcup_{a\in M}[a]$ for a finite $M\subseteq\kkp{X}$.
Now, if $v\in U^*$,
then $v\res_X=u\res_X$ for some $u\in U\subseteq Z$.
Hence, $\scp{a}{u}>0$ for some $a\in M$.
Since $a_i=0$ for $i\notin X$, we also have $\scp{a}{v}=\scp{a}{u}>0$, hence $v\in Z$.

Similarly, let $Z\in\Op\kkp{X})$, $Z\subseteq U$.
We verify that $Z\subseteq U_*$. It suffices to consider the case
when $Z=\bigcap_{a\in M}[a]$ for a finite $M\subseteq\kkp{X}$. Now, let
$v\in Z\subseteq U$ and $u\in E$ with $v\res_X=u\res_X$. Then $\scp{a}{v}=\scp{a}{u}>0$ for every $a\in M$,
hence $u\in U$. By the definition, $v\in U_*$.
\end{proof}

\begin{lemma}\label{interpol}
Let  $X,Y\in\fin{I}$, $U\in\Op\kkp{X}$, $V\in\Op\kkp{Y}$, and $U\subseteq V$. Then $U\subseteq W\subseteq V$ for some $W\in\Op\kkp{X\cap Y}$.
\end{lemma}

\begin{proof}
Since quantifier elimination does not add
new variables, the upper shadow $U^* $ on $\kkp{Y}$ belongs to $\kkp{X\cap Y}$.
From $U\subseteq V$ we obtain $U^*\subseteq V$, so we can set $W=U^*$.
\end{proof}

\begin{proposition}\label{P:OpEFS}
The \dzlat{s}~$\Op{E}$ and~$\Ops{E}$ are both finitely separable.
\end{proposition}

\begin{proof}
By Lemmas \ref{shad} and \ref{interpol},  $\Op\kkp{I}$ is the strong amalgam (cf. Definition~\ref{D:StrAmalg}) of all $\Op\kkp{X}$ for $X\in\fin{I}$.
Since~$\kk$ is countable, so are all $\Op\kkp{X}$, which are therefore finitely separable (cf. Proposition~\ref{P:BasicFS}\eqref{ctbleFS}).
By Proposition~\ref{P:StrAmFS}, $\Op\kkp{I}$ is finitely separable.
By Proposition~\ref{P:BasicFS}\eqref{ConvFS}, so is thus $\Ops\kkp{I}=(\Op\kkp{I})\setminus\set{E}$.
\end{proof}

Note that by the Baker-Bernau-Madden duality (cf. Baker~\cite{Baker1968}, Bernau~\cite{Bern1969}, Madden \cite[Ch.~III]{MaddTh}, and also Wehrung \cite[page~13]{RAlg} for a summary), $\Ops\kkp{I}$ is isomorphic to the lattice of all principal \lidl{s} of the free $\kk$-vector lattice on~$I$.
Hence, by applying Theorem~\ref{T:MonCev} and Remark~\ref{Rk:MonCev}, we obtain the following.

\begin{corollary}\label{C:OpEFS}
Let~$E$ be a left vector space over a countable totally ordered division ring~$\kk$.
Then the \dzlat{s}~$\Op{E}$ and~$\Ops{E}$ are both monotone-Cevian.
\end{corollary}

If we wanted only ``existence of a monotone deviation'' instead of ``monotone-Cevian'' in Corollary~\ref{C:OpEFS}, then it would be sufficient to use Theorem~\ref{T:DevFS} instead of Theorem~\ref{T:MonCev} (the latter itself following from Freese and Nation \cite[Theorem~1]{FreNat2015}).

\begin{corollary}\label{C:fvl}
For any set~$I$ and any countable totally ordered division ring~$\kk$, 
the lattice of all principal \lidl{s} of the free $\kk$-vector lattice on~$I$ is monotone-Cevian.
\end{corollary}

If $\kk=\QQ$ is the ordered field of rationals, then the free $\QQ$-vector lattice~$F$ on~$I$ is the divisible hull of the free Abelian \lgrp~$G$ on~$I$, thus the lattices of principal \lidl{s} of~$F$ and~$G$ are both isomorphic to~$\Ops\QQ^{(I)}$.

\begin{corollary}\label{C:flg}
For any set~$I$, 
the lattice of all principal \lidl{s} of the free  Abelian \lgrp\ on~$I$ is monotone-Cevian.
\end{corollary}

\section{An Archimedean vector lattice counterexample}\label{S:lgrpCX}

Since the lattice of all principal \lidl{s} of any countable Abelian \lgrp~$G$ is countable, thus finitely separable, the Belluce map of~$G$ has in that case an isotone section; hence~$\Idc{G}$ is monotone-Cevian (cf. Corollary~\ref{C:MonCev2}).
In this section we will verify that that observation does not extend to the uncountable case, by constructing an Archimedean vector lattice~$G$ with strong unit, of cardinality~$\aleph_1$, whose lattice of all principal \lidl{s} does not have any monotone deviation (thus, \emph{a fortiori}, no isotone section for the Belluce map; see Theorem~\ref{T:MonCev}).

We begin with a refinement of Shanin's classical $\gD$-Lemma (see for example Jech \cite[Theorem~9.18]{Jech2003}).

\begin{lemma}\label{L:kfree}
Let $\Phi\colon\go_1\to [\go_1]^{<\go}$ and let~$X$ be a cofinal subset of $\go_1$.
Then there are ascending $\go_1$-sequences $\vecm{\tau_i}{i<\go_1}$ and $\vecm{\ga_i}{i<\go_1}$ of ordinals in~$\go_1$ such that for every $k<\go_1$,
 \begin{enumerater}
 \item\label{gaktat}
 $\ga_k\in X$ and $\tau_k<\ga_k<\tau_{k+1}$;
 
 \item\label{tkcont}
 $\tau_k=\sup_{j<k}\tau_j$ if~$k$ is a limit ordinal;
 
 \item\label{bdPhi}
$\Phi(\ga_k)\subseteq\tau_0\cup(\tau_{k+1}\setminus\tau_k)$.
 
 \end{enumerater}

\end{lemma}

\begin{proof}
By the $\gD$-Lemma, there are a cofinal subset~$Y$ of~$X$ and a set~$K$ such that
 \begin{equation}\label{Eq:PhigD}
 \Phi(\ga)\cap\Phi(\gb)=K\text{ whenever }\ga,\gb\in Y
 \text{ and }\ga\neq\gb\,.
 \end{equation}
Pick $\tau_0<\go_1$ such that $K\subseteq\tau_0$.
Let $i<\go_1$ and set
 \[
 i^*\eqdef\begin{cases}i\,,&\text{if }i\text{ is a limit ordinal}\,,\\
 i+1\,,&\text{otherwise}\,.
 \end{cases}
 \]
We argue by induction on~$i$.
Suppose having constructed ascending sequences $\vecm{\tau_k}{k<i^*}$ and $\vecm{\ga_k}{k<i}$ satisfying \eqref{gaktat}--\eqref{bdPhi} whenever $k<i$.
If~$i$ is either~$0$ or a successor, then~$\tau_i$ is already defined.
If~$i$ is a limit, set $\tau_i\eqdef\sup_{j<i}\tau_j$; this takes care of~\eqref{tkcont} at level~$i$.
By the cofinality statement on~$Y$, the countability of~$\tau_i$, and~\eqref{Eq:PhigD}, there exists $\ga_i\in Y$ such that $\ga_i>\tau_i$ and $\Phi(\ga_i)\cap\tau_i\subseteq K$.
Pick $\tau_{i+1}<\go_1$ such that $\tau_{i+1}>\ga_i$ and $\Phi(\ga_i)\subseteq\tau_{i+1}$.
Since $K\subseteq\tau_0$, \eqref{gaktat} and~\eqref{bdPhi} at level~$i$ follow.
\end{proof}

By the Baker-Bernau-Madden duality, the free  $\QQ$-vector lattice~$F(\go_1)$ on $\go_1+\nobreak1$ is isomorphic to the vector sublattice~$F$ of $\QQ^{\QQ^{\go_1+1}}$ generated by all canonical projections $x_{\ga}\colon\QQ^{\go_1+1}\twoheadrightarrow\QQ$ for $\ga\in\go_1+1$.
(Here, unlike in Section~\ref{S:DevCX}, we  use $\QQ^{\go_1+1}$ instead of~$\QQ^{(\go_1+1)}$.)
Every element of~$F$ has the form $\dot{f}\vecm{x_{\ga}}{\ga\in\go_1+1}$ for some vector lattice term~$\dot{f}$.
By the above-cited duality, $\Idc F$ is isomorphic to
the $0$-sublattice of the powerset lattice of~$\QQ^{\go_1}$ generated by the cozero sets $\setm{u\in\QQ^{\go_1+1}}{g(u)\neq 0}$, the latter corresponding to the principal \lidl~$\seq{g}$.
This leads immediately to the following assertion.

\begin{lemma}\label{L:free0}
For any $g,h\in F^+$, the inequality $\seq{g}\le\seq{h}$ holds \pup{within~$\Idc F$} if{f} $h^{-1}\set{0}\subseteq g^{-1}\set{0}$.
\end{lemma}

Now we introduce our counterexample.
Let us denote
 \begin{multline*}
 \gO_M\eqdef\{u\in\QQ^{\go_1} \mid
 0\leq u_{\ga}\leq 1
 \text{ whenever }\ga\in M\\
 \text{ and }u_{\gc}\leq 2u_{\gb}\text{ whenever }
 \gb,\gc\in M\text{ and }
 0<\gc<\gb\}
 \end{multline*}
whenever $M\subseteq\go_1$, and set $\gO\eqdef \gO_{\go_1}$.
In the sequel, let~$G$ be the vector sublattice of~$\QQ^{\gO}$ generated by all 
canonical projections $p_{\ga}\colon\gO\to\QQ$
for $\ga\in\go_1$ together with the constant function $\gO\to\QQ$, $u\mapsto1$, which we shall conveniently
denote~$p_{\go_1}$ even though it is not a ``projection''.

\begin{lemma}\label{L:free1}
Let $f_0\colon\setm{p_{\ga}}{\ga\in\go_1}\to\QQ$ satisfy $0\leq f_0(p_{\ga})\leq 1$ whenever $\ga<\go_1$
and $f_0(p_{\gc})\leq 2f_0(p_{\gb})$ whenever $0<\gc<\gb<\go_1$.
Then~$f_0$ can be extended to a unique \lhom\ $G\to\QQ$ mapping~$p_{\go_1}$ to~$1$.
\end{lemma}

\begin{proof}
We only need to verify the existence part.
The vector $w\eqdef\vecm{f_0(p_{\ga})}{\ga\in\go_1}$ belongs to~$\gO$.
The evaluation morphism $e_w\colon g\mapsto g(w)$ is the desired extension.
\end{proof}

\begin{lemma}\label{L:ExtG2G}
Let $f_0\colon\setm{p_{\ga}}{\ga\in\go_1}\to G$ satisfy $0\leq f_0(p_{\ga})\leq p_{\go_1}$ whenever $\ga<\nobreak\go_1$
and $f_0(p_{\gc})\leq 2f_0(p_{\gb})$ whenever $0<\gc<\gb<\go_1$.
Then~$f_0$ can be extended to a unique $\ell$-en\-do\-mor\-phism of~$G$ mapping~$p_{\go_1}$ to itself.
\end{lemma}

\begin{proof}
We use the fact that~$G$ is a vector sublattice of~$\QQ^{\gO}$ \emph{via} the evaluation morphisms~$e_u$ for $u\in\gO$.
By Lemma~\ref{L:free1}, each map $f_0^u\eqdef e_uf_0$ can be extended to an \lhom\ 
$f_u\colon G\to\QQ$ such that $f_u(p_{\go_1})=1$.
Then the product homomorphism $f\eqdef\prod_{u\in\gO}f_u$ extends $f_0$ and $f(p_{\go_1})=p_{\go_1}$.
\end{proof}

\begin{lemma}\label{L:ZM2Z}
Let~$M$ be a subset of~$\go_1$.
Then for every $u\in\gO_M$ there exists $v\in\gO$ such that $u\res_{M}=v\res_{M}$.
\end{lemma}

\begin{proof}
Set $M{\uparrow}\xi\eqdef\setm{\ga\in M}{\ga\geq\xi}$ whenever $\xi<\go_1$, and define
  \[
 v_{\xi}\eqdef\begin{cases}
 u_{\tau}\,,&\text{if } M{\uparrow}\xi\ne\es\text{ and } \tau=\min M{\uparrow}\xi\,,\\
1\,,&\text{otherwise}\,,
 \end{cases}
 \quad\text{whenever }\xi<\go_1\,.
 \]
It is easy to see that $v\in\gO$.
\end{proof}

\begin{lemma}\label{L:IdcG}
 For any $g,h\in G^+$, the inequality $\seq{g}\le\seq{h}$ holds \pup{within~$\Idc{G}$} if{f} 
$h^{-1}\set{0}\subseteq g^{-1}\set{0}$.
\end{lemma}

\begin{proof}
We verify the nontrivial direction.
Suppose that $h^{-1}\set{0}\subseteq g^{-1}\set{0}$.
Since $\set{g,h}\subseteq G$ there are a finite subset~$M$ of~$\go_1$ and vector lattice terms~$\dot{g}$ and~$\dot{h}$ such that
 \[
 g=\dot{g}\vecm{p_{\ga}}{\ga\in M\cup\set{\go_1}}\text{ and }
 h=\dot{h}\vecm{p_{\ga}}{\ga\in M\cup\set{\go_1}}\,.
 \]
We can assume $M\ne\es$.
Define elements~$d$, $g_1$, $h_1$ of~$F$ as
 \[
 d\eqdef\bigvee_{\ga\in M}(-x_{\ga})\vee
 \bigvee_{\ga\in M}(x_{\ga}-x_{\go_1})
 \vee\bigvee_{\gc,\gb\in M,\ 0<\gc<\gb}
 (x_{\gc}-2x_{\gb})\,,
 \]
then $g_1\eqdef\dot{g}\vecm{x_{\ga}}{\ga\in M\cup\set{\go_1}}\vee d$ and
$h_1\eqdef\dot{h}\vecm{x_{\ga}}{\ga\in M\cup\set{\go_1}}\vee d$.
We denote 
 \begin{multline*}
 \bgO_M\eqdef\{u\in\QQ^{\go_1+1} \mid
 0\leq u_{\ga}\leq u_{\go_1}
 \text{ whenever }\ga\in M\\
 \text{ and }u_{\gc}\leq 2u_{\gb}\text{ whenever }
 \gb,\gc\in M\text{ and }
 0<\gc<\gb\}.
 \end{multline*}

\begin{sclaim}
Let $u\in\bgO_M$. Then 
$g_1(u)\geq0$ and $h_1(u)\geq0$, and, further, $h_1(u)=0$ implies that $g_1(u)=0$.
\end{sclaim}

\begin{scproof}
The assumption $u\in\bgO_M$ implies $d(u)\leq 0$.
The case $u_{\go_1}<0$ is excluded as $0\leq u_\alpha\leq u_{\go_1}$ for $\alpha\in M\ne\es$.
If $u_{\go_1}=0$ then $u_\alpha=0$ for every $\alpha\in M$, so $g_1(u)=h_1(u)=0$ and we are done. Suppose now that  $u_{\go_1}>0$.
Consider the element $u'\eqdef u_{\go_1}^{-1}\cdot u$.
Then~$u'\res_{\go_1}$ belongs to~$\gO_M$ and
it follows from Lemma~\ref{L:ZM2Z} that $u'\res_{M}=v\res_{M}$ for some $v\in\gO$. 
We obtain
\[ \dot{g}\vecm{u_{\ga}}{\ga\in M\cup\set{\go_1}}=u_{\go_1} \dot{g}\vecm{u'_{\ga}}{\ga\in M\cup\set{\go_1}} =
u_{\go_1}g(v).
\]
(For the last equality notice that $u'_{\go_1}=1=p_{\go_1}(v)$.)
Since $g\in G^+$ we get $g(v)\ge 0$.
Since $d(u)\leq 0$, we obtain
 \[
 g_1(u)= u_{\go_1}g(v)\vee d(u)= u_{\go_1}g(v)\geq0\,,
 \]
and similarly,
\[ 
h_1(u)=u_{\go_1}h(v)\geq0\,.
\]
Hence $h_1(u)=0$ implies $h(v)=0$, thus (since $v\in\gO$) $g(v)=0$, and thus $g_1(u)=\nobreak0$.
\end{scproof}

On the other hand, for any $u\in\QQ^{\go_1+1}\setminus\bgO_M$, $d(u)>0$ and hence $g_1(u),h_1(u)>0$.
Therefore, it follows from the Claim above that $\set{g_1,h_1}\subseteq F^+$ and $h_1^{-1}\set{0}\subseteq g_1^{-1}\set{0}\subseteq\gO_M$.
By Lemma~\ref{L:free0}, the latter relation entails $\seq{ g_1}\leq\seq{h_1}$ (within~$\Idc{F}$), that is, $g_1\leq mh_1$ for some $m>0$.
Then for every $u\in\gO$ (so $u\conc(1)\in\gO_M$), $g(u)=g_1(u\conc(1))\leq mh_1(u\conc(1))=mh(u)$, thus $g\leq mh$, and thus $\seq{g}\leq\seq{h}$ (within~$\Idc{G}$).
\end{proof}

\begin{lemma}\label{L:pscom}
For every $\ga\in\go_1\setminus\set{0}$ and every positive $c\in\QQ$, $\seq{(p_0-cp_{\ga})^+}$ and $\seq{(cp_{\ga}-p_0)^+}$ are pseudocomplements of each other within~$\Idc{G}$.
\end{lemma}

\begin{proof}
Trivially $\seq{(p_0-cp_{\ga})^+}\wedge\seq{(cp_{\ga}-p_0)^+}=0$.

Now let $t\in G^+$.
Assume first that $\seq{t}\wedge\seq{(p_0-cp_{\ga})^+}=0$.
By way of contradiction, suppose that $\seq{t}\nleq\seq{(cp_{\ga}-p_0)^+}$.
By Lemma~\ref{L:IdcG}, there exists $z\in\gO$ such that $t(z)>0$ whereas $cz_{\ga}\leq z_0$.
{F}rom $\seq{t}\wedge\seq{(p_0-cp_{\ga})^+}=0$ it thus follows that $z_0\leq cz_{\ga}$, thus $z_0=cz_{\ga}$.
We separate cases.

\begin{itemize}
\item
If $z_0<1$, let $y\in\gO$ be defined by $y_0=z_0+\eps$ whereas $y_{\gb}=z_{\gb}$ whenever $0<\gb<\go_1$, with~$\eps>0$ chosen small enough to ensure $z_0+\eps\leq 1$ and  $t(y)>0$.
This is possible since~$t(y)$ depends only on a finite number of components of~$y$.
Then $y_0-cy_{\ga}=(z_0-cz_{\ga})+\eps=\eps>0$, in contradiction with $\seq{t}\wedge\seq{(p_0-cp_{\ga})^+}=0$.

\item
If $z_0=1$, so $z_{\ga}>0$, let $y\in\gO$ be defined by $y_0=z_0$ whereas $y_{\gb}=(1-\eps)z_{\gb}$ whenever $0<\gb<\go_1$, with~$\eps>0$ chosen small enough to ensure $t(y)>0$.
Then $y_0-cy_{\ga}=(z_0-cz_{\ga})+\eps z_{\ga}=\eps z_{\ga}>0$, in contradiction with $\seq{t}\wedge\seq{(p_0-cp_{\ga})^+}=0$.
\end{itemize}

Assume now that $\seq{t}\wedge\seq{(cp_{\ga}-p_0)^+}=0$.
By way of contradiction, suppose that $\seq{t}\nleq\seq{(p_0-cp_{\ga})^+}$.
By Lemma~\ref{L:IdcG}, there exists $z\in\gO$ such that $t(z)>0$ whereas $cz_{\ga}\geq z_0$.
{F}rom $\seq{t}\wedge\seq{(cp_{\ga}-p_0)^+}=0$ it thus follows that $cz_{\ga}\leq z_0$, thus $z_0=cz_{\ga}$.
We separate cases.

\begin{itemize}
\item
If $z_0>0$, let $y\in\gO$ be defined by $y_0=(1-\eps)z_0$ whereas $y_{\gb}=z_{\gb}$ whenever $0<\gb<\go_1$, with~$\eps>0$ chosen small enough to ensure $t(y)>0$.
Then\linebreak $cy_{\ga}-y_0=(cz_{\ga}-z_0)+\eps z_0=\eps z_0>0$, in contradiction with\linebreak $\seq{t}\wedge\seq{(cp_{\ga}-p_0)^+}=0$.

\item
If $z_0=0$, so $z_{\ga}=0$, let $y\in\gO$ be defined by $y_0=z_0$ whereas $y_{\gb}=\min\set{z_{\gb}+\eps,1}$ whenever $0<\gb<\go_1$, with~$\eps>0$ chosen small enough to ensure $\eps<1$ and $t(y)>0$.
Then $cy_{\ga}-y_0=(cz_{\ga}-z_0)+c\eps=c\eps>0$, in contradiction with $\seq{t}\wedge\seq{(cp_{\ga}-p_0)^+}=0$.
\qed
\end{itemize}
\renewcommand{\qed}{}
\end{proof}

\begin{theorem}\label{T:NoIsoAnti}
The vector lattice~$G$ is Archimedean with strong unit.
Furthermore, no deviation on~$\Idc{G}$ can be either left isotone or right antitone.
In fact, for every deviation~$\sd$ on~$\Idc{G}$ there are ordinals~$\ga$, $\gb$, $\ga'$, $\gb'$ such that $0<\ga<\gb<\go_1$, $0<\ga'<\gb'<\go_1$, $\seq{p_0}\sd\seq{p_{\gb}}\nleq\seq{p_0}\sd\seq{p_{\ga}}$, and $\seq{p_{\ga'}}\sd\seq{p_0}\nleq\seq{p_{\gb'}}\sd\seq{p_0}$.
\end{theorem}

\begin{proof}
Since~$G$ is an \lsgrp\ of a power of~$\QQ$ it is Archimedean.
Moreover, $p_{\go_1}$ is a strong unit of~$G$.

Let~$\sd$ be a deviation on~$\Idc{G}$ and suppose first that
 \begin{equation}\label{Eq:AntitoneAssumption}
 \seq{p_0}\sd\seq{p_{\gb}}
 \leq\seq{p_0}\sd\seq{p_{\ga}}\quad
 \text{whenever }0<\ga<\gb<\go_1\,.
 \end{equation}
Whenever $0<\ga<\go_1$ let $s_{\ga}\in G^+$
with
$\seq{s_{\ga}}=\seq{p_0}\sd\seq{p_{\ga}}$.
There exists a finite subset~$\Phi(\ga)$ of~$\go_1$ such that $s_{\ga}$ belongs to the vector sublattice of $G$ generated by
$\setm{p_{\gb}}{\gb\in\Phi(\ga){\cup\set{p_{\go_1}}}}$.

It is well known that the elements $\seq{(p_0-mp_{\ga})^+}$ (with $m\in\go$) form a coinitial subset in the set
$\setm{\bx\in\Idc{G}}{\seq{p_0}\leq\seq{p_{\ga}}\vee\bx}$. 
Hence, $\seq{(p_0-mp_{\ga})^+}\le
\seq{s_{\ga}}$ for some positive $m$.
Similarly, $\seq{(p_{\ga}-np_0)^+}\leq\seq{p_{\ga}}\sd\seq{p_0}$ for some positive~$n$.
Hence,
 \[
 \seq{s_{\ga}}\wedge\seq{(p_{\ga}-np_0)^+}\leq
 (\seq{p_0}\sd\seq{p_{\ga}})
 \wedge(\seq{p_{\ga}}\sd\seq{p_0})=0\,.
 \]
By Lemma~\ref{L:pscom}, $\seq{s_{\ga}}\le
\seq{(np_0-p_{\ga})^+}$.
We have thus obtained the following.

\setcounter{claim}{0}

\begin{claim}\label{Cl:Encadrtga}
Whenever $0<\ga<\go_1$ there are positive integers~$m$, $n$ such that
 \[
 \seq{(p_0-mp_{\ga})^+}\leq
 \seq{s_{\ga}}\le\seq{(np_0-p_{\ga})^+}\,.
 \]
\end{claim}

For all positive integers~$m$, $n$ denote
 \[
 X_{m,n}\eqdef
 \setm{\ga\in\go_1\setminus\set{0}}{\seq{(p_0-mp_{\ga})^+}\leq
\seq{s_{\ga}}\le\seq{(np_0-p_{\ga})^+}}\,.
\]
By Claim~\ref{Cl:Encadrtga}, at least one of those sets must be cofinal in~$\go_1$;
we choose such~$(m,n)$ and set $X\eqdef X_{m,n}$.
Further, we pick a positive integer~$k$ such that $2^{k-1}>mn$.

Applying Lemma \ref{L:kfree} with~$X$ and~$\Phi$ defined above, we find corresponding ordinals~$\tau_i$ and~$\ga_i$ (for $i<\go_1$).
We may assume that $\tau_0>0$.
Let us keep only those~$\tau_i$, $\ga_i$ with $0\leq i<k$, and reset $\tau_k\eqdef\go_1$.
Let $f\colon\setm{p_{\gc}}{\gc\in\go_1}\to G$ be defined as follows:
 \begin{equation}\label{Eq:Themapf}
 f(p_{\gc})\eqdef\begin{cases}
 p_0\,,&\text{if }\gc\in\tau_0\,,\\
 p_{\ga_i}\,,&\text{if }\gc\in\tau_{i+1}\setminus\tau_i
 \end{cases}\quad
 \text{whenever }\gc\in\go_1\,.
 \end{equation}
Since $0\leq f(p_{\ga})\leq p_{\go_1}$ for any~$\ga\in\go_1$ and since $f(p_{\gc})\leq 2f(p_{\gd})$ whenever $0<\gc\le\gd$, the map~$f$ can, by Lemma~\ref{L:ExtG2G}, be
extended to an $\ell$-en\-do\-mor\-phism of~$G$, which we shall also denote by~$f$, such that $f(p_{\go_1})=p_{\go_1}$.
The definition of~$\Phi$ ensures that each~$f(s_{\ga_i})$ belongs to the vector sublattice of~$G$ generated by $\set{p_0,p_{\ga_i},p_{\go_1}}$.

\begin{claim}\label{Cl:Ineqfsgak-i}
The inequality $\seq{f(s_{\ga_{k-i}})} \geq \seq{(2^{i-1}p_0-mp_{\ga_{k-i}})^+}$ holds whenever\linebreak $1\leq i\leq k$.
\end{claim}

\begin{cproof}
We proceed by induction. For $i=1$ our statement follows from
the inequality $\seq{(p_0-mp_{\ga_{k-1}})^+}\le
\seq{s_{\ga_{k-1}}}$.
(Notice that $f(p_0)=p_0$ and $f(p_{\ga_j})=p_{\ga_j}$ for every $j$, thus
$f((p_0-mp_{\ga_{k-1}})^+)=(p_0-mp_{\ga_{k-1}})^+$.)

Now let $i>1$ and suppose for contradiction that $\seq{f(s_{\ga_{k-i}})} \ngeq \seq{(2^{i-1}p_0-mp_{\ga_{k-i}})^+}$.
By Lemma~\ref{L:IdcG}, there is $ z\in\gO$ such that
$f(s_{\ga_{k-i}})( z)\leq 0$ and $2^{i-1}z_0>mz_{\ga_{k-i}}$.
We define $y\in\gO$ by
 \begin{equation}\label{Eq:3casesy}
 y_{\gb}\eqdef\begin{cases}
 z_0\,,&\text{if }\gb=0\,,\\
 z_{\ga_{k-i}}\,,&\text{if }0<\gb\leq\ga_{k-i}\,,\\
 \frac{1}{2}z_{\ga_{k-i}}\,,&\text{if }\gb>\ga_{k-i}\,,
 \end{cases}
 \quad\text{whenever }\gb<\go_1\,.
 \end{equation}
Since $f(s_{\ga_{k-i}})$ belongs to the vector sublattice of~$G$ generated by $\set{p_0,p_{\ga_{k-i}},p_{\go_1}}$, we get $f(s_{\ga_{k-i}})(y)=f(s_{\ga_{k-i}})(z)\leq 0$.
Our assumption~\eqref{Eq:AntitoneAssumption} entails the inequality $\seq{s_{\ga_{k-i}}}\ge\seq{s_{\ga_{k-i+1}}}$, hence 
$\seq{f(s_{\ga_{k-i}})}\ge\seq{f(s_{\ga_{k-i+1}})}$ and therefore
$ f(s_{\ga_{k-i+1}})(y)\leq\nobreak 0$. Further,
 \[
 2^{i-2}y_0-my_{\ga_{k-i+1}}=
 \frac{1}{2}(2^{i-1}z_0-mz_{\ga_{k-i}})>0\,,
 \]
which contradicts the induction hypothesis $\seq{f(s_{\ga_{k-i+1}})} \geq \seq{(2^{i-2}p_0-mp_{\ga_{k-i+1}})^+}$.
\end{cproof}

For $i=k$, Claim~\ref{Cl:Ineqfsgak-i} yields $\seq{f(s_{\ga_0})} \geq \seq{(2^{k-1}p_0-mp_{\ga_0})^+}$.
On the other hand,
$\ga_0\in X_{m,n}$ implies  $\seq{s_{\ga_0}}\le\seq{(np_0-p_{\ga_0})^+}$, hence
$\seq{f(s_{\ga_0})}\le\seq{f((np_0-p_{\ga_0})^+)}=\seq{(np_0-p_{\ga_0})^+}$.
Therefore,
 \[
 \seq{(2^{k-1}p_0-mp_{\ga_0})^+}\le
 \seq{(np_0-p_{\ga_0})^+}\,.
 \]
However, this inequality does not hold.
Indeed, by Lemma~\ref{L:ZM2Z} there exists $z\in\gO$ with $z_0=1/n$ and $z_{\ga_0}=1$.
Then $nz_0-z_{\ga_0}=0$, while
 \[
 2^{k-1}z_0-m z_{\ga_0}=2^{k-1}/n-m>0
 \]
according to the choice of~$k$.
This contradiction shows that~\eqref{Eq:AntitoneAssumption} cannot hold.

The ``left isotone'' case is handled similarly.
Suppose that
 \begin{equation}\label{Eq:IsotoneAssumption}
 \seq{p_{\ga}}\sd\seq{p_0}
 \leq\seq{p_{\gb}}\sd\seq{p_0}\quad
 \text{whenever }0<\ga<\gb<\go_1\,.
 \end{equation}
Whenever $0<\ga<\go_1$ let $t_{\ga}\in G^+$
with
$\seq{t_{\ga}}=\seq{p_{\ga}}\sd\seq{p_0}$.
There exists a finite subset~$\Phi(\ga)$ of~$\go_1$ such that $t_{\ga}$ belongs to the vector sublattice of $G$ generated by
$\setm{p_{\gb}}{\gb\in\Phi(\ga)\cup\set{\go_1}}$.

The proof of the following claim is then, \emph{mutatis mutandis}, identical to the one of Claim~\ref{Cl:Encadrtga}, thus we shall omit it.

\begin{claim}\label{Cl:Encadrsga}
Whenever $0<\ga<\go_1$ there are positive integers~$m$, $n$ such that
 \[
 \seq{(p_{\ga}-mp_0)^+}\leq
 \seq{t_{\ga}}\leq\seq{(np_{\ga}-p_0)^+}\,.
 \]
\end{claim}
For all positive integers~$m$, $n$ we now denote
 \[
 Y_{m,n}\eqdef
 \setm{\ga\in\go_1\setminus\set{0}}{\seq{(p_{\ga}-mp_0)^+}\leq
\seq{t_{\ga}}\leq\seq{(np_{\ga}-p_0)^+}}\,.
\]
By Claim~\ref{Cl:Encadrsga}, at least one of those sets must be cofinal in~$\go_1$;
we choose such~$(m,n)$ and set $Y\eqdef Y_{m,n}$.
Again, we pick a positive integer~$k$ such that $2^{k-1}>mn$.

Applying Lemma \ref{L:kfree} with~$Y$ and~$\Phi$ defined above, we find corresponding ordinals~$\tau_i$ and~$\ga_i$ (for $i<\go_1$) with $\tau_0>0$.
Again, we keep only those~$\tau_i$, $\ga_i$ with $0\leq i<k$, and reset $\tau_k\eqdef\go_1$.
The map $f\colon\setm{p_{\gc}}{\gc\in\go_1}\to G$ defined as in~\eqref{Eq:Themapf} can again be extended to a unique $\ell$-en\-do\-mor\-phism of~$G$ such that $f(p_{\go_1})=p_{\go_1}$, and the definition of~$\Phi$ ensures that each~$f(t_{\ga_i})$ belongs to the vector sublattice of~$G$ generated by $\set{p_0,p_{\ga_i},p_{\go_1}}$.

\begin{claim}\label{Cl:Ineqftgak-i}
The inequality $\seq{f(t_{\ga_{k-i}})} \leq \seq{(np_{\ga_{k-i}}-2^{i-1}p_0)^+}$ holds whenever\linebreak $1\leq i\leq k$.
\end{claim}

\begin{cproof}
We proceed by induction, following the lines of the proof of Claim~\ref{Cl:Ineqfsgak-i}.
For $i=1$ our statement follows from
the inequality $\seq{t_{\ga_{k-1}})} \leq \seq{(np_{\ga_{k-1}}-p_0)^+}$.

Now let $i>1$ and suppose for contradiction that $\seq{f(t_{\ga_{k-i}})} \nleq \seq{(np_{\ga_{k-i}}-2^{i-1}p_0)^+}$.
By Lemma~\ref{L:IdcG}, there is $z\in\gO$ such that
$f(t_{\ga_{k-i}})( z)>0$ whereas $nz_{\ga_{k-i}}\leq 2^{i-1}z_0$.
We define~$y\in\gO$ as in~\eqref{Eq:3casesy} and observe again that $f(t_{\ga_{k-i}})(y)=f(t_{\ga_{k-i}})(z)>0$.
Our assumption~\eqref{Eq:IsotoneAssumption} entails $\seq{t_{\ga_{k-i}}}\leq\seq{t_{\ga_{k-i+1}}}$, hence 
$\seq{f(t_{\ga_{k-i}})}\leq\seq{f(t_{\ga_{k-i+1}})}$ and therefore
$ f(t_{\ga_{k-i+1}})(y)>0$.
Further,
 \[
 ny_{\ga_{k-i+1}}-2^{i-2}y_0=
 \frac{1}{2}(nz_{\ga_{k-i}}-2^{i-1}z_0)\leq0\,,
 \]
which contradicts the induction hypothesis $\seq{f(t_{\ga_{k-i+1}})} \leq \seq{(np_{\ga_{k-i+1}}-2^{i-2}p_0)^+}$.
\end{cproof}

For $i=k$, Claim~\ref{Cl:Ineqftgak-i} yields $\seq{f(t_{\ga_0})} \leq \seq{(np_{\ga_0}-2^{k-1}p_0)^+}$.
On the other hand,
$\ga_0\in Y_{m,n}$ implies $\seq{(p_{\ga_0}-mp_0)^+}\leq\seq{t_{\ga_0}}$, hence also
$\seq{(p_{\ga_0}-mp_0)^+}\leq\seq{f(t_{\ga_0})}$.
Therefore,
 \[
\seq{(p_{\ga_0}-mp_0)^+}\le
 \seq{(np_{\ga_0}-2^{k-1}p_0)^+}\,.
 \]
However, this inequality does not hold.
Indeed, by Lemma~\ref{L:ZM2Z} there exists $z\in\gO$ with $z_0=2^{1-k}$ and $z_{\ga_0}=1/n$.
Then $nz_{\ga_0}-2^{k-1}z_0=0$, while
 \[
 z_{\ga_0}-mz_0=1/n-m2^{1-k}>0
 \]
according to the choice of~$k$.
This contradiction shows that~\eqref{Eq:IsotoneAssumption} cannot hold, thus concluding the proof of Theorem~\ref{T:NoIsoAnti}.
\end{proof}

Finally, we would like to state some open problems.

\begin{problem}
Is every finitely separable \cn\ lattice Cevian?
\end{problem}

By Theorem~\ref{T:DevFS}, every finitely separable \cn\ lattice has a monotone deviation; we do not know whether that deviation can be made Cevian.
So far, the only known examples of non-Cevian \cn\ lattices are due to Wehrung \cite[\S~7]{Ceva} and \cite[\S~7]{MVRS}, the former lattice satisfying the additional property of having \emph{countably based differences}.
We do not know whether any of those lattices is finitely separable.

\begin{problem}
Is it possible that a \cn\ lattice has a deviation monotone in one variable, but none in both variables?
\end{problem}

Recall that for the \lgrp~$G$ constructed above, the lattice $\Idc{G}$ does not have a deviation monotone in either variable.

\begin{problem} Is it possible, for an Abelian \lgrp~$G$,  that $\Idc{G}$ has a monotone deviation but no isotone Belluce section?
\end{problem}

Notice that in the proof of Corollary~\ref{C:flg} we actually constructed  an isotone Belluce section for any free Abelian \lgrp.

\section{Declarations}

{\bf Authors' Contribution:}

Both authors contributed to the results in this paper.
\vskip5mm

{\bf Conflict of Interest:}

On behalf of all authors, the corresponding author states that there is no conflict of interest.\vskip5mm

{\bf Availability of Data and Materials:}

Data sharing is not applicable to this article as no datasets were generated or analysed during the current study.\vskip5mm

{\bf Funding:}

The first author was supported by Slovak VEGA grant 1/0152/22.


\providecommand{\noopsort}[1]{}\def\cprime{$'$}
  \def\polhk#1{\setbox0=\hbox{#1}{\ooalign{\hidewidth
  \lower1.5ex\hbox{`}\hidewidth\crcr\unhbox0}}} \def\cprime{$'$}
  \def\cprime{$'$} \def\cprime{$'$} \def\cprime{$'$} \def\cprime{$'$}
\providecommand{\bysame}{\leavevmode\hbox to3em{\hrulefill}\thinspace}
\providecommand{\MR}{\relax\ifhmode\unskip\space\fi MR }
\providecommand{\MRhref}[2]{%
  \href{http://www.ams.org/mathscinet-getitem?mr=#1}{#2}
}
\providecommand{\href}[2]{#2}

\end{document}